\documentclass[11pt]{article}
\usepackage[latin1]{inputenc}
\usepackage{amsmath,amsthm,amssymb}
\usepackage{amsfonts}
\usepackage{amsmath,amsthm,amssymb,amscd}
\usepackage{latexsym}
\usepackage{color}
\usepackage{graphicx}
\usepackage{mathrsfs}

\makeatletter \@addtoreset{equation}{section} \makeatother

\newtheorem{theorem}{Theorem}[section]

\newtheorem{lemma}{Lemma}[section]
\newtheorem{remark}{Remark}[section]

\begin{document}
\title{\bf  On critical exponential Kirchhoff systems on the Heisenberg group} 
\author{\bf Shiqi Li$^{a}$, Sihua Liang$^{a}$ and Du\v{s}an D. Repov\v{s}$^{b,c,d,}$\thanks{Corresponding author (D.D. Repov\v{s}) } 
\thanks{{\it E-mail address:} lishiqi59@126.com (S. Li),   liangsihua@163.com (S. Liang), dusan.repovs@guest.arnes.si (D.D. Repov\v{s})}
\thanks{{ORCID ID:} https://orcid.org/0000-0002-7853-2862  (S. Li),   https://orcid.org/0000-0002-4260-7762  (S. Liang), https://orcid.org/0000-0002-6643-1271 (D.D. Repov\v{s})}
\\
$^{\small\mbox{a}}${\small College of Mathematics, Changchun Normal University, Changchun, 130032,  P.R. China }\\[-0.3cm]
$^{\small\mbox{b}}${\small Faculty of Education, University of Ljubljana, Ljubljana, 1000, Slovenia}\\[-0.3cm]
$^{\small\mbox{c}}${\small Faculty of Mathematics and Physics, University of Ljubljana, Ljubljana, 1000, Slovenia}\\[-0.3cm]
$^{\small\mbox{d}}${\small Institute of Mathematics, Physics and Mechanics, Ljubljana, 1000, Slovenia}
}
\date{}
\maketitle
\begin{abstract}
In this paper,  existence of solutions  is
established for
critical exponential  Kirchhoff systems on the Heisenberg group by
using the variational method. The novelty of our paper is that not
only the nonlinear term has critical exponential growth, but also that
Kirchhoff function  covers the  degenerate case. Moreover, our result
is new even for the Euclidean case.


\end{abstract}

\section{Introduction}\label{sec1}
In this paper, we are interested in  critical exponential
Kirchhoff systems on the Heisenberg group $\mathbb{H}^{n}$:
\begin{equation}\label{e1.1}
\begin{cases}-K_1(\int_{\Omega}|\nabla_{\mathbb{H}^{n}} u|^{Q}d\xi)\Delta_{Q} u =\lambda f_1(\xi, u, v),  & \hbox{for} \  \xi\in \Omega, \\
-K_2(\int_{\Omega}|\nabla_{\mathbb{H}^{n}} v|^{Q}d\xi)\Delta_{Q} v =\lambda f_2(\xi, u, v), & \hbox{for} \  \xi\in  \Omega,  \\
u = v =0, &  \hbox{for} \ \xi\in \partial \Omega,\end{cases}
\end{equation}
where $\Delta_{Q}$ is the $Q$-Laplacian operator on the Heisenberg group  $\mathbb{H}^{n}$, defined by
$$\Delta_{Q}(\cdot)=\operatorname{div}_{\mathbb{H}^{n}}(\left|\nabla_{\mathbb{H}^{n}}(\cdot)\right|_{\mathbb{H}^{n}}^{Q-2}
\nabla_{\mathbb{H}^{n}}(\cdot)),$$ $\Omega$ is a bounded open smooth
 subset of the Heisenberg group $\mathbb{H}^{n},$ and $\lambda >0$
is a positive parameter.

There are already several interesting papers devoted to the study of
the Heisenberg group  $\mathbb{H}^{n}$. For example, Pucci and
Temperini \cite{pu2} studied the existence of entire nontrivial
solutions of $(p, q)$ critical systems on the Heisenberg group
$\mathbb{H}^{n}$, by using the variational methods and the
concentration-compactness principle. In Pucci and Temperini
\cite{pu1}, they accomplished the conclusions of Pucci and Temperini
\cite{pu2} and worked out some kind of elliptic systems involving
critical nonlinearities and Hardy terms on the Heisenberg group
$\mathbb{H}^{n}$. There are additional interesting results in 
Liang and Pucci
 \cite{liang},  Pucci \cite{pu4},
Pucci \cite{ pu3} and  Pucci and Temperini \cite{ pu5}.

In the Euclidean case, Kirchhoff-type  problems have attracted wave after wave
of scholars. In $1883$, Kirchhoff \cite{ki} established a
model given by the  hyperbolic equation
$$\rho\frac{\partial ^2u}{\partial
t^2}-\left(\frac{p_0}{h}+\frac{E}{2L} \int_0^L\left|\frac{\partial
u}{\partial x}\right|^2dx\right) \frac{\partial ^2u}{\partial
x^2}=0,$$
where parameters $\rho$, $p_0$, $h$, $E$, $L$ are constants  with
some physical meaning, which extends the classical
  D'Alembert wave equation for free vibrations of elastic strings.
  In particular, Kirchhoff equation models also appear in physical and biological systems.
 We refer the reader to Alves et al. \cite{al}
 for more details.  After Kirchhoff's work,  Figueiredo and Severo $\cite{fise}$ studied the following problem
$$\left\{\begin{array}{cl}
&-m\left(\int_{\Omega}|\nabla u|^{2} d x\right) \Delta u=f(x, u)  \text { in } \Omega, \\
&u=0  \text { on } \partial \Omega,
\end{array}\right.$$
and addressed the existence of ground state solutions of the
problems on $\mathbb{R}^{2}$ by using the minimax techniques with
the Trudinger-Moser inequality.  Mingqi et al. \cite{MR}
studied the existence and multiplicity of solutions for a class of
perturbed fractional Kirchhoff type problems with singular
exponential nonlinearity.  Alves and Boudjeriou \cite{al2}
obtained the existence of a nontrivial solution for a class of
nonlocal problems by using the dynamical methods. For several
interesting results recovering the Kirchhoff-type problems, we refer
to Ambrosio et al. \cite{am}, Caponi and Pucci \cite{capu},
Mingqi et al. \cite{XMTZ}, and Pucci et al. \cite{PXZ1}, and the
references therein. 

Recently, some authors have focused their attention to the problem
with critical exponential growth in the Euclidean case, see
Aouaoui \cite{ao},
Albuquerque et al. \cite{de},
Moser \cite{Mo}, and
Trudinger \cite{Tr}.
 In the Heisenberg group  $\mathbb{H}^{n}$ case,
 Cohn and Lu \cite{co} have established a new version of the Trudinger-Moser inequality:
Let $\Omega \subset \mathbb{H}^{n}$ and assume that $|\Omega|<\infty$ and
$0<\alpha \leq \alpha_{Q}.$
Then
\begin{equation}\label{e1.2}
\sup _{u \in W_{0}^{1, Q}(\Omega),\left\|\nabla_{\mathbb{H}^{n}}
u\right\|_{L^{Q}(\Omega)} \leq 1} \frac{1}{|\Omega|} \int_{\Omega}
e ^{\alpha|u(\xi)|^{\frac{Q}{Q-1}} }d \xi \leq
C_{0}<\infty,
\end{equation}
where $C_{0}>0$ is a constant which depends only on $Q=2 n+2$.
Moreover, $$\alpha_{Q}=Q \sigma_{Q}^{1 /(Q-1)} \quad \text { and }
\quad \sigma_{Q}=\int_{\rho(z, t)=1}|z|^{Q} d \mu.$$ Deng and Tian
\cite{den} have established the existence of nontrivial solutions
for the non-degenerate Kirchhoff elliptic system with nonlinear term
have critical exponential growth. Although the study of critical
Kirchhoff-type problems is more meaningful, there are some  authors
working on the degenerate Kirchhoff problem. From a physical point
of view, the fact that $M(0)=0$ means that the base tension of the
string is zero, which is a very realistic model. To the best of our
knowledge, the existence results for system  \eqref{e1.1} in
the degenerate case are not yet known for the Heisenberg group
$\mathbb{H}^{n}$.

For these reasons, we mainly consider the critical exponential
Kirchhoff systems \eqref{e1.1} on the Heisenberg group. We say that $f_i$
satisfies critical exponential growth at $+\infty$ provided that there exists
$\alpha_0 > 0$ such that
\begin{eqnarray}\label{e1.3}\lim\limits _{|(u, v)|
\rightarrow+\infty} \frac{|f_{i}(\xi, u, v)|}{e^{{\alpha|(u,
v)|^{\frac{Q}{Q-1}}}}}=
\begin{cases}0 &
\text { uniformly on } \xi \in \Omega,$ for all $\alpha>\alpha_{0}; \\
+\infty & \text { uniformly on } \xi \in \Omega,$ for all $\alpha<\alpha_{0}.\end{cases}\end{eqnarray} In view of the critical
exponential
 growth of the nonlinear terms $f_{i}$, we work out the problem of the "lack of
compactness" by using  a new version of
the Trudinger-Moser
inequality for the Heisenberg group $\mathbb{H}^{n}$.

Throughout the paper, Kirchhoff-type functions $K_i$ and  $f_i$ will
satisfy the following conditions:
\begin{itemize}
\item[$(\mathcal{K})$]
$(\mathcal{K}_1)$ There exist $q\in[1,n)$ and $\sigma \in [1, q),$ satisfying
$$ \sigma\mathcal{K}_{i}(t):=\sigma\int_{0}^{t} K_{i}(s) ds\geq
K_{i}(t)t,
\quad
\hbox{for all}
\quad
 t\geqslant 0.$$
$(\mathcal{K}_2)$ There exist $k_1,k_{2} > 0$ such that $K_{i}(t)
\geq k_i t^{\sigma-1}$, where $K(0) = 0,$ for all $t \geqslant 0$.
\item[$(\mathcal{F})$]
$(F_1)$   $\lim _{w \rightarrow 0}
\frac{f_{i}(\xi,w)}{|w|^{Q-1}}=0,$ for all $w=(u, v)$ and
$|w|=\sqrt{u^{2}+v^{2}}$.

$(F_2)$ $F \in C^{1}(\mathbb{R} \times \mathbb{R}, \mathbb{R})$  and
there exists $\mu >\sigma Q$ such that $$ 0<\mu F(\xi,w) \leqslant
\nabla F(\xi,w) w, \ \hbox{ for all} \ w \in \mathbb{R}^{2}, \ \hbox{
where} \ \nabla F=\left(f_{1}, f_{2}\right).$$ $(F_3)$ $ \liminf _{w
\rightarrow 0^{+}} \frac{F(\xi, w)}{|w|^\mu}=: \eta>0.$
\end{itemize}
\begin{remark}\label{rmk1.0}
Note that some typical examples of functions $K_{i}:\mathbb{R}_{+}
\rightarrow \mathbb{R}_{+}$ which are nondecreasing and satisfy conditions
 $(\mathcal{K}_1)$ and $(\mathcal{K}_2)$, are given by
$$K_{i}(t)=a_{i}+b_{i} t^{\sigma-1},
\
\hbox{ for all}
\
t \in \mathbb{R},
\
\hbox{ where}
\
a
\in \mathbb{R}, b \in \mathbb{R}
\
\hbox{ and}
\
a+b>0.$$
  Clearly, conditions
$(\mathcal{K}_1)$ and $(\mathcal{K}_2)$  also cover the degenerate case, that
is when $a=0$.
\end{remark}
\begin{remark}\label{rmk1.1}
By condition $\left(F_{2}\right)$, we can easily get
that
$F(\xi, w) / |w|^\mu$ is nondecreasing for all $w>0$. Thus, for all $w
\geqslant 0, $ we obtain $ F(\xi, w) \geqslant \eta |w|^\mu$ by
invoking condition $\left(F_{3}\right)$.
\end{remark}

The main result of this paper is as follows.
\begin{theorem}\label{the1.1} Assume that conditions $(\mathcal{K})$ and $(\mathcal{F})$ are satisfied
and that
 the nonlinear terms $f_{i}$ have critical
 exponential growth. Then system $\eqref{e1.1}$ has at least one nontrivial solution for all $\lambda>0$, when $\eta>0$ from condition $\left(F_{3}\right)$ is large enough.
\end{theorem}
In conclusion, we describe the structure of the paper.
In Section \ref{sec2} we collect all  necessary preliminaries.
In Section \ref{sec3} we study the mountain pass geometry.
In Section \ref{sec4} we verify the compactness condition.
Finally, in Section \ref{sec5} we present the proof of our main result.

\section{Preliminaries}\label{sec2}

We begin by recalling some key facts about the Heisenberg group $\mathbb{H}^{n}$, i.e. a Lie group of
 topological dimension $2 n+1$ and the background manifold
$\mathbb{R}^{2 n+2}$. We define
 $$ \xi \circ
\xi^{\prime}=\tau_{\xi}\left(\xi^{\prime}\right)=\left(x+x^{\prime},
y+y^{\prime}, t+t^{\prime}+2\left(x^{\prime} y-y^{\prime}
x\right)\right), \ \hbox{ for all} \
 \ \xi=(x,y,z), \xi^{\prime}=(x',y',z') \in \mathbb{H}^{n}.$$
Furthermore,
$$\delta_{s}(\xi)=\left(s x, s y, s^{2} t\right),
\
\hbox{ where}
\
s>0,$$
gives a natural
group of dilations on $\mathbb{H}^{n}$, so
  $$\delta_{s}\left(\xi_{0} \circ \xi\right)=\delta_{s}\left(\xi_{0}\right) \circ \delta_{s}(\xi).$$
The Jacobian determinant of dilatations $\delta_{s}:
\mathbb{H}^{n} \rightarrow \mathbb{H}^{n}$ is constant for all
$\xi=(x, y, t) \in \mathbb{H}^{n}$ and $\delta_{s}=\mathbb{R}^{2
n+2}$.
Next, $$
 B_{R}\left(\xi_{0}\right)=\left\{\xi \in \mathbb{H}^{n}:
d_{K}\left(\xi, \xi_{0}\right)<R\right\} $$ 
denotes
 the Kor\'{a}nyi open
ball with radius $R,$ centered at $\xi_{0}$ (see Leonardi and Masnou \cite{le}).

The
horizontal gradient $\nabla_{H}=(X, Y)$ and $$
X=\frac{\partial}{\partial x}+2 y \frac{\partial}{\partial t},
Y=\frac{\partial}{\partial y}-2 x \frac{\partial}{\partial t}$$ give
rise to the Lie algebra of left-invariant vector fields on
$\mathbb{H}^{n}$. Finally, $\Delta_{H}
u=\operatorname{div}_{H}\left(\nabla_{H} u\right)$ represents the
Kohn-Laplacian $\Delta_{H}$. In addition,
 the
 degenerate elliptic operator $\Delta_{H}$ satisfies Bony's
maximum principle (see Bony \cite{bo2}).

Next, we define the classical Sobolev space $W_{0}^{1, Q}(\Omega)$ as the closure of $\mathcal{C}_{0}^{\infty}(\Omega)$ with
 respect  to the norm
$$\|u\|:=\left\|\nabla_{H^{n}}
u\right\|_{W_{0}^{1, Q}}. $$
Let $$W_{0}^{1, Q}\left(\Omega, \mathbb{R}^{2}\right):=W_{0}^{1, Q}(\Omega) \times
W_{0}^{1, Q}(\Omega),$$   endowed by the norm $$ \|(u,
v)\|_{W_{0}^{1, Q}\left(\Omega, \mathbb{R}^{2}\right)}=(\|u\|_{W_{0}^{1, Q}(\Omega)}^{Q}+\|v\|_{W_{0}^{1,
Q}(\Omega)}^{Q})^{1/Q}.$$ System \eqref{e1.1} is variational and the
corresponding energy functional $I_{\lambda}: W_{0}^{1, Q}\left(\Omega, \mathbb{R}^{2}\right) \rightarrow
\mathbb{R}$ is given by
$$
I_{\lambda}(u,v):=\frac{1}{Q}\mathcal{K}_{1}\left( \int_{\Omega}\left|\nabla_{\mathbb{H}^{n}} u\right|^{Q} d \xi\right)+\frac{1}{Q} \mathcal{K}_{2} \left(\int_{\Omega}\left|\nabla_{\mathbb{H}^{n}} v\right|^{Q} d \xi\right)-\lambda \int_{\Omega} F(\xi, u,v) d \xi.
$$
By $\left(F_{1}\right),\left(F_{2}\right)$ and \eqref{e1.3},  there exists for
$\varepsilon>0,r>Q$,  a constant $C=C(\varepsilon,r)>0$
such that
\begin{equation}\label{e2.1}
|f_{1}(\xi, w)|+|f_{2} (\xi, w)| \leq
\varepsilon|w|^{Q-1}+C|w|^{r-1} e ^{\alpha|w|^{Q^{\prime}}}, \quad
\text { for all } \alpha>\alpha_{0}.
\end{equation}
Consequently, we have
\begin{equation}\label{e2.2}
F(\xi, w)  \leq \frac{\varepsilon}{\theta Q}|w|^{Q}+C|w|^{r} e
^{\alpha|w|^{Q^{\prime}}}, \quad \text { for all } \alpha>\alpha_{0},
\end{equation}
where $Q^{\prime}=\frac{Q}{Q-1}$.
Since
\begin{equation}\label{e2.3}
(a+b)^{m} \leq 2^{m-1}\left(a^{m}+b^{m}\right), \ \text { for all }
a, b \geq 0, \ m>0,
\end{equation}
we can obtain the following inequality
$$|w|^{Q^{\prime}}=\left(|u|^{2}+|v|^{2}\right)^{\frac{Q^{\prime}}{2}}
\leq 2^{Q^{\prime}-1}\left(|u|^{Q^{\prime}}+|v|^{Q^{\prime}}\right).
$$ Thus, invoking the H\"{o}lder inequality and \eqref{e1.2}, we get
$$
\begin{aligned}\int_{\Omega} e ^{\alpha|w|^{Q^{\prime}}} d \xi & \leq \int_{\Omega} e^{\left(\alpha 2^{Q^{\prime}-1}|u|^{Q^{\prime}}\right)} e ^{\left(\alpha 2^{Q^{\prime}-1}|v|^{Q^{\prime}}\right)} d \xi\\
&\leq\left(\int_{\Omega}e ^{\alpha 2^{Q^{\prime}}|u|^{Q^{\prime}}}d \xi\right)^{1 / 2}\left(\int_{\Omega}e ^{\alpha 2^{Q^{\prime}}|v|^{Q^{\prime}} }d \xi\right)^{1 / 2}<\infty .
\end{aligned}
$$
Moreover, we can conclude that   $I_{\lambda}\in C^{1}( W_{0}^{1, Q}\left(\Omega, \mathbb{R}^{2}\right), \mathbb{R})$ is well-defined and that the derivative of $I_{\lambda}$ is
\begin{eqnarray*}
\left\langle I_{\lambda}^{\prime}(u, v),(\varphi,
\psi)\right\rangle&=& K_{1}(\|u\|^{Q})
\int_{\Omega}\left|\nabla_{\mathbb{H}^{n}}
u\right|_{\mathbb{H}^{n}}^{Q-2} \nabla_{\mathbb{H}^{n}} u
\nabla_{\mathbb{H}^{n}} \varphi d\xi+
K_{2}\left(\|v\|^{Q}\right) \int_{\Omega}\left|\nabla_{\mathbb{H}^{n}} v\right|_{\mathbb{H}^{n}}^{Q-2}  \nabla_{\mathbb{H}^{n}} v \nabla_{\mathbb{H}^{n}} \psi d \xi \\
&&\ -\lambda \int_{\Omega} f_{1}(\xi, u, v)\varphi d\xi-\lambda
\int_{\Omega}f_{2}(\xi, u, v)\psi d \xi,\  \hbox{for all} \
 (u,
v),(\varphi, \psi) \in W_{0}^{1, Q}\left(\Omega, \mathbb{R}^{2}\right).
\end{eqnarray*}
Therefore the solutions of system \eqref{e1.1} coincide with the
critical points of $I_{\lambda}$.
\begin{lemma}\label{lem2.1}
Suppose that condition $\left(F_{1}\right)$ is satisfied and
that
$$t \geqslant0,
\ \
\alpha>\alpha_{0},
\ \
\|w\|\leq\rho,
\
\hbox{ with}
\quad
 \alpha
\rho^{Q^{\prime}}<\frac{\alpha_{Q}}{2Q^{\prime}}.$$
Then there
exists $C=C(t, \alpha,\rho)>0$ such that $$ \int_{\Omega}e^{\alpha
|w|^{Q^{\prime}}}|w|^{t} d \xi\leqslant C\|w\|^{t}.$$
\end{lemma}
\begin{proof}
Due to the H\"{o}lder inequality, we get
$$ \int_{\Omega}e
^{\alpha|w|^{Q^{\prime}}}|w|^{t} d \xi \leq\left(\int_{\Omega} e ^{q
\alpha|w|^{Q^{\prime}}} d \xi\right)^{\frac{1}{q}}\|w\|_{p t}^{t},
\
\hbox{where}
\
p,q>1,\frac{1}{p}+\frac{1}{q}=1,
\
 pt\geq Q.$$
Since
$$\|w\| \leqslant \rho
\
\hbox{ with}
\
 \alpha
\rho^{Q^{\prime}}<\frac{\alpha_{Q}}{2^{Q^{\prime}}},$$
there exists
$q>1$
 such that
$$
2^{^{Q^{\prime}}} q \alpha \rho^{Q^{\prime}} \leq
\alpha_{Q}.
$$
 By virtue of \eqref{e2.3}, one has
$$
\begin{aligned}
\int_{\Omega} e^{q\alpha|w|^{Q^{\prime}}} d \xi & \leq \int_{\Omega} e ^{\left(2^{Q^{\prime}-1} q \alpha\left(|u|^{Q^{\prime}}+|v|^{Q^{\prime}}\right)\right)} d \xi
%\\&
 \leq\left(\int_{\Omega} e ^{\left(2^{Q^{\prime}} q \alpha|u|^{Q^{\prime}}\right)} d \xi\right)^{\frac{1}{2}}\left(\int_{\Omega} e ^{\left(2^{Q^{\prime}} q \alpha|v|^{Q^{\prime}}\right)}d \xi\right)^{\frac{1}{2}} \\
& \leq\left(\int_{\Omega} e ^{\left(2^{Q^{\prime}} q \alpha \rho^{Q^{\prime}}\left(\frac{|u|}{\|w\|}\right)^{Q^{\prime}}\right)} d \xi\right)^{\frac{1}{2}}\left(\int_{\Omega} e ^{2^{Q^{\prime}} q \alpha \rho^{Q^{\prime}}\left(\frac{|v|}{\|w\|}\right)^{Q^{\prime}}}d \xi\right)^{\frac{1}{2}}
\leq C,
\end{aligned}
$$
which implies that the conclusion of Lemma \ref{lem2.1} is valid.\end{proof}

For other background information we refer the reader to the comprehensive monograph by Papageorgiou et al. \cite{PRR}.

\section{Mountain Pass Geometry}\label{sec3}
In this section  we shall prove that  $I_{\lambda}$ satisfies the
mountain pass geometry.
\begin{lemma}\label{lem2.2} 
 Suppose that conditions $(\mathcal{K})$, $(\mathcal{F})$ are satisfied  and
 that
  $f_{i}$ have
 exponential critical growth. Then the following properties hold:

$\left(I_{1}\right)$ There exist $\iota>0$ and $\kappa>0$ such that $I_{\lambda}(u, v) \geqslant \iota,$ for all  $\|(u, v)\|=\kappa$.

$\left(I_{2}\right)$ There exists $(e,e) \in  W_{0}^{1, Q}\left(\Omega, \mathbb{R}^{2}\right)$ with $\|(e,
e)\|>\kappa$ such that $I_{\lambda}(e, e)<0$.
\end{lemma}
\begin{proof}First, we prove assertion $(I_1).$ If $r>\sigma Q$, $\alpha>\alpha_{0},$ and $0<\varepsilon<\min \left\{k_{1}, k_{2}\right\},$ then by virtue of $(F_{1})$ and \eqref{e2.2}, we have

$$I_{\lambda}(u, v) \geqslant \frac{1}{\sigma Q}\left(\min
\left\{k_{1}, k_{2}\right\}-\lambda \varepsilon\right)\|(u,
v)\|^{\sigma Q}-C
_{1}\lambda\int_{\Omega}e^{\alpha|w|^{Q^{\prime}}}|(u, v)|^{r} d\xi,
\ \hbox{ for all} \ (u, v) \in   W_{0}^{1, Q}\left(\Omega,
\mathbb{R}^{2}\right).$$

 By  Lemma \ref{lem2.1}  and the following inequality
$$\|(u, v)\|=\kappa< \left(\frac{\alpha_{Q}}{2\alpha
Q^{\prime}}\right)^{\frac{1}{Q^{\prime}}},$$  we obtain $$I_{\lambda}(u, v)
\geqslant \frac{1}{\sigma Q}\left(\min \left\{k_{1},
k_{2}\right\}-\lambda \varepsilon\right) \kappa^{\sigma Q}-C
_{2}\kappa^{r}.$$ Next, we choose $$\kappa<
\left(\frac{\alpha_{Q}}{2\alpha
Q^{\prime}}\right)^{\frac{1}{Q^{\prime}}}$$ so small that
 $$\frac{1}{\sigma Q}\left(\min \left\{k_{1},
k_{2}\right\}-\lambda \varepsilon\right) -C_{2} \kappa^{r-\sigma
Q}>0.$$ This implies that $$I_{\lambda}(u, v) \geqslant
\iota:=\kappa^{\sigma Q}\left[\frac{1}{\sigma Q}\left(\min
\left\{k_{1}, k_{2}\right\}-\lambda \varepsilon\right)-C_{2}
\kappa^{r-\sigma Q}\right], \ \hbox{ for all} \ \|(u, v)\|=\kappa.$$

Next, we prove assertion $\left(I_{2}\right)$. Let $\psi\in C_{0
}^{\infty}(B_{R}(\xi_{0}))$ be such that $\psi \geqslant 0$ on
$B_{R}(\xi_{0})$ and  let $K=\operatorname{supp}(\psi)$.
Invoking  condition $(\mathcal{K})$, we  obtain
 $$\mathcal{K}(t) \leq c_{i} t^{\sigma}+d_{i},\ \text { for all } \ t
\geqslant 0, \ \hbox{where} \ c_{i}, d_{i}>0, i= 1, 2.$$ By virtue
of Remark \ref{rmk1.1}, we get
$$
I_{\lambda}(t \psi, t \psi) \leqslant \frac{t^{\sigma Q}}{Q}\|\psi\|^{\sigma Q}\left(c_{1}+d_{1}\right)+\frac{1}{Q}\left(c_{2}+d_{2}\right)-\lambda \eta t^{\mu} \int_{\Omega} \psi^{\mu}d\xi.
$$
Since $\mu>\sigma Q,$ we can conclude that $$I_{\lambda}(t \psi, t
\psi) \rightarrow-\infty, \ \hbox{ as} \ t \rightarrow+\infty.$$
Therefore, we  get the claim by using $e:=t \psi$ for a
sufficiently large $t>0$.
\end{proof}

\section{Compactness Condition}\label{sec4}
In this section, we shall prove the following compactness condition.
\begin{lemma}\label{lem2.3}
 Suppose that the following inequality holds
\begin{equation}\label{e2.4}
c<\frac{\nu_{0}(\mu-\sigma Q) \alpha_{Q}^{\sigma(Q-1)}}{ \alpha_{0}^{\sigma(Q-1)} Q\mu
\sigma^{\sigma Q+1}},
\
\hbox{where}
\
\nu_{0}:=\min \left\{k_{1},
k_{2}\right\}.
\end{equation}
Then there exists   a
$(\mathrm{PS})_{c}$ sequence for $I_{\lambda}$ $\left\{\left(u_{n}, v_{n}\right)\right\}$ such that the functional $I_{\lambda}$
satisfies the Palais-Smale condition at level $c$.
\end{lemma}
\begin{proof} First, let $$\left\{\left(u_{n}, v_{n}\right)\right\}\subset W_{0}^{1, Q}\left(\Omega, \mathbb{R}^{2}\right)$$ be a $(PS)_{c}$ sequence for $I_{\lambda}$. If $\inf_{n\in\mathbb N}\|w_{n}\|= 0$, then
 $$w_{n} \rightarrow 0
 \
 \hbox{   in}
 \
 W_{0}^{1, Q}\left(\Omega, \mathbb{R}^{2}\right),
 \ \hbox{ as}
 \
 n \rightarrow \infty.$$ Therefore, we shall use $\inf_{n\in\mathbb
N}\|w_{n}\|>0$  in the sequel.

By conditions $(\mathcal{K})$ and
$\left(F_{2}\right)$, we have
\begin{equation}\label{e2.7}
\begin{aligned}
c+o_{n}(1)\left\|\left(u_{n}, v_{n}\right)\right\|& \geqslant
I_{\lambda}\left(u_{n}, v_{n}\right)-\frac{1}{\mu}
I_{\lambda}^{\prime}\left(u_{n}, v_{n}\right)\left(u_{n}
v_{n}\right) \\&\geqslant \left(\frac{1}{\sigma Q}-\frac{1}{\mu}\right) K_{1}\left(\left\|u_{n}\right\|^{Q}\right)\left\|u_{n}\right\|^{Q}+\left(\frac{1}{\sigma Q}-\frac{1}{\mu}\right)K_{2}\left(\left\|v_{n}\right\|^{Q}\right)\left\|v_{n}\right\|^{Q} \\
&+\frac{\lambda}{\mu} \int\left[\nabla F\left(\xi ,u_{n}, v_{n}\right)\left(u_{n}, v_{n}\right)-\mu F\left(\xi,u_{n}, v_{n}\right)\right] d \xi,
\end{aligned}
\end{equation}
which combined with conditions $\left(H_{1}\right)$ and $\left(F_{2}\right),$ implies
$$
c+o_{n}(1)\left\|\left(u_{n}, v_{n}\right)\right\|\geqslant \left(\frac{\mu-\sigma Q}{\sigma Q\mu}\right)
\nu_{0}\left\|\left(u_{n}, v_{n}\right)\right\|^{\sigma Q},
\
\hbox{where}
\
\nu_{0}:=\min \left\{k_{1}, k_{2}\right\}.$$

Hence,
\eqref{e2.7} implies that  the sequence $\left\{\left(u_{n},
v_{n}\right)\right\}$ is bounded on $W_{0}^{1, Q}\left(\Omega, \mathbb{R}^{2}\right)$ and that $$\limsup _{n
\rightarrow+\infty}\left\|\left(u_{n}, v_{n}\right)\right\|^{\sigma Q}
\leqslant \frac{\sigma Q\mu }{(\mu-\sigma Q) \nu_{0}} c.$$

Next, let $\left(u_{0}, v_{0}\right)\in W_{0}^{1, Q}\left(\Omega, \mathbb{R}^{2}\right)$ be such that $\left(u_{n}, v_{n}\right)
\rightharpoonup\left(u_{0}, v_{0}\right)$ weakly in $W_{0}^{1, Q}\left(\Omega, \mathbb{R}^{2}\right)$.
We shall prove the
convergence
\begin{equation}\label{e2.9}
\int_{\Omega} f_{1}\left(\xi,u_{n},
v_{n}\right)\left(u_{n}-u_{0}\right) \mathrm{d} \xi \rightarrow
0,
\
%\mbox{and} \quad
 \int_{\Omega} f_{2}\left(\xi,u_{n},
v_{n}\right)\left(u_{n}-u_{0}\right) \mathrm{d} \xi \rightarrow 0, \
\hbox{as} \
 n \rightarrow+\infty.
\end{equation}

Let $\varepsilon>0, \alpha>\alpha_{0}, s>1$, and
$s^{\prime}=s /(s-1).$
By \eqref{e2.1} and
H\"{o}lder's inequality, we have
$$
\left|\int_{\Omega} f_{1}\left(\xi,u_{n}, v_{n}\right)\left(u_{n}-u_{0}\right) d\xi\right|
\leqslant \varepsilon\left\|\left(u_{n}, v_{n}\right)\right\|_{Q}^{Q-1}\left\|u_{n}-u_{0}\right\|_{Q^{\prime}}
+C\left(\int_{\Omega}e^{\alpha s|w_{n}|^{Q^{\prime}}}d \xi\right)^{1 / s}\left\|u_{n}-u_{0}\right\|_{s^{\prime}}.
$$

Thus, invoking  \eqref{e2.3} and the compactness of the embedding $  W_{0}^{1,
Q}\left(\Omega\right) \hookrightarrow
L^{s^{\prime}}\left(\Omega\right)$, we obtain
\begin{equation}\label{e2.10}
\begin{aligned}
\left|\int_{\Omega} f_{1}\left(\xi,u_{n},
v_{n}\right)\left(u_{n}-u_{0}\right)d\xi\right|&\leqslant \varepsilon C+C\left( \int_{\Omega}e^{s\alpha 2^{Q^{\prime}-1} (|u_{n}|^{Q^{\prime}}+|v_{n}|^{Q^{\prime}})} \mathrm{d} \xi\right)^{\frac{1}{s}} o_{n}(1)\\& \leqslant
\varepsilon C+C\left(\int_{\Omega} e ^{2^{Q^{\prime}} s \alpha
|u_{n}|^{Q^{\prime}}} d
\xi\right)^{\frac{1}{2s}}\left(\int_{\Omega}
e^{2^{Q^{\prime}}s \alpha|v_{n}|^{Q^{\prime}}}d
\xi\right)^{\frac{1}{2s}}.
\end{aligned}
\end{equation}
Due to $$ \int_{\Omega} e ^{2^{Q^{\prime}} s \alpha
|u_{n}|^{Q^{\prime}}} d \xi \leqslant \int_{\Omega} e
^{\left(\sigma^{Q^{\prime}} s \alpha
|w_{n}|^{Q^{\prime}}\left(\frac{|u_{n}|}{\|w_{n}\|}\right)^{Q^{\prime}}\right)}
d \xi,$$ we can choose $\delta>0$ such that
$$\sigma^{Q^{\prime}}\alpha_{0}\left\|\left(u_{n},
v_{n}\right)\right\|^{Q^{\prime}} \leqslant \alpha_{Q}-\delta, \
\hbox{ for sufficiently large} \ n \in \mathbb{N}.$$ Therefore, we
have $$\sigma ^{Q^{\prime}}s\alpha\left\|\left(u_{n},
v_{n}\right)\right\|^{Q^{\prime}} <\alpha_{Q}, \ \hbox{ for sufficiently large} \ n \in \mathbb{N},$$ where $\alpha>\alpha_{0}$ is close to
$\alpha_{0}$ and  $s>1$ is close to $1$.

It then follows from  \eqref{e1.2} that it suffices to show  that the following holds$$
\int_{\Omega} e ^{\sigma^{Q^{\prime}} s \alpha |u_{n}|^{Q^{\prime}}}
d \xi\leqslant C.$$ Similarly, we can get$$\int_{\Omega} e
^{\sigma^{Q^{\prime}}\sigma \alpha|v_{n}|^{Q^{\prime}}}d \xi\leqslant C.$$
As above, we get \eqref{e2.9}.
Finally, we define $$ \Phi(u, v):=\frac{1}{Q}
\mathcal{K}_{1}\left(\|u\|^{Q}\right)+\frac{1}{Q}
\mathcal{K}_{2}\left(\|v\|^{Q}\right),$$where
$\mathcal{K}_{i}$ is the convexity,
when $\left(\mathcal{K}\right)$ holds.

Due to  weak lower semicontinuity, we have
\begin{equation}\label{e2.11}
\frac{1}{Q}
\mathcal{K}_{1}\left(\left\|u_{0}\right\|^{Q}\right)+\frac{1}{Q}\mathcal{K}_{2}\left(\left\|v_{0}\right\|^{Q}\right)
\leqslant \frac{1}{Q} \liminf _{n \rightarrow+\infty}
\mathcal{K}_{1}\left(\left\|u_{n}\right\|^{Q}\right)+\frac{1}{Q}
\liminf _{n \rightarrow+\infty}
\mathcal{K}_{2}\left(\left\|v_{n}\right\|^{Q}\right).
\end{equation}
 Moreover,  by virtue of    \eqref{e2.9} and convexity of $\Phi(u, v)$, we  obtain
$$
\begin{aligned}
\Phi\left(u_{0}, v_{0}\right)-\Phi\left(u_{n}, v_{n}\right)&\geqslant \Phi^{\prime}\left(u_{n}, v_{n}\right)\left(u_{0}-u_{n}, v_{0}-v_{n}\right) \\&=
I_{\lambda}^{\prime}\left(u_{n}, v_{n}\right)\left(u_{0}-u_{n},
v_{0}-v_{n}\right)+\lambda \int_{\Omega} f_{1}\left(\xi,u_{n},
v_{n}\right)\left(u_{0}-u_{n}\right) \mathrm{d} \xi \\ & \quad
+\lambda \int_{\Omega} f_{2}\left(\xi,u_{n},
v_{n}\right)\left(v_{0}-v_{n}\right) \mathrm{d}\xi.
\end{aligned}
$$
Therefore, we
have $$\Phi\left(u_{0}, v_{0}\right)+o_{n}(1) \geqslant
\Phi\left(u_{n}, v_{n}\right)$$ and we get
\begin{eqnarray*}
\begin{aligned}
\frac{1}{Q}\mathcal{K}_{1}
\left(\left\|u_{0}\right\|^{Q}\right)+\frac{1}{Q}\mathcal{K}_{2}\left(\left\|v_{0}\right\|^{Q}\right)
&\geqslant \limsup _{n \rightarrow+\infty} \Phi\left(u_{n},
v_{n}\right)
\\
&\geqslant \frac{1}{Q} \liminf _{n \rightarrow+\infty}
\mathcal{K}_{1}\left(\left\|u_{n}\right\|^{Q}\right)+\frac{1}{Q}
\liminf _{n \rightarrow+\infty}
\mathcal{K}_{2}\left(\left\|v_{n}\right\|^{Q}\right).
\end{aligned}
\end{eqnarray*}
This fact together with  \eqref{e2.11}, yields the contradiction.
Therefore,
$$
\frac{1}{Q} \mathcal{K}_{1}\left(\left\|u_{n}\right\|^{Q}\right)
\rightarrow \frac{1}{Q}
\mathcal{K}_{1}\left(\left\|u_{0}\right\|^{Q}\right) \quad \text {
and } \quad \frac{1}{Q}
\mathcal{K}_{2}\left(\left\|v_{n}\right\|^{Q}\right) \rightarrow
\frac{1}{Q} \mathcal{K}_{2}\left(\left\|v_{0}\right\|^{Q}\right), \
\hbox{as} \ n \rightarrow+\infty.$$

We can conclude that $$\left\|u_{n}\right\|^{Q}
\rightarrow\left\|u_{0}\right\|^{Q} \ \hbox{and} \
\left\|v_{n}\right\|^{Q} \rightarrow \left\|v_{0}\right\|^{Q},$$
since $K_1(t)$ and $K_2(t)$ are increasing for $t>0,$ as $n
\rightarrow+\infty$. Therefore, $\left(u_{n}, v_{n}\right)
\rightarrow\left(u_{0}, v_{0}\right)$ strongly in $E,$ and the proof
is complete.
\end{proof}

\section{Proof of Theorem \ref{the1.1}}\label{sec5}

 We claim that
\begin{equation}\label{e3.1}
c^{*}:=\inf _{\gamma \in \Gamma} \max _{t \in[0,1]}
I_{\lambda}(\gamma(t)),
\end{equation}
where
$$
\Gamma:=\left\{\gamma \in C([0,1], W_{0}^{1, Q}\left(\Omega, \mathbb{R}^{2}\right)): \gamma(0)=(0,0) \text { and } I_{\lambda}(\gamma(1))<0\right\}.
$$
If we assume
that  \eqref{e3.1} holds, then Lemmas \ref{lem2.2} and
\ref{lem2.3} and the Mountain pass lemma yield the existence of
nontrivial critical points of $I_{\lambda}$.
\begin{lemma}\label{lem3.1}
Assume that
\begin{equation}\label{e3.2}
\eta>\max \left\{\eta_{1},\frac{\left(Q
\eta_{1}\right)^{\frac{\mu}{Q}}}{\mu}\left(\frac{\sigma^{\sigma Q+1}\alpha_{0}^{\sigma(Q-1)}
\lambda\pi(\mu-Q)}{\nu_{0}\alpha_{Q}^{\sigma(Q-1)}(\mu-\sigma Q)}\right)^{\frac{\mu-Q}{Q}}\right\},
\end{equation}
where $$\eta_{1}:=[\mathcal{K}_{1}(2\varepsilon^{Q}\pi)+\mathcal{K}_{2}(2\varepsilon^{Q}\pi)] /(Q \lambda \pi)
\
\hbox{ and}
\ \ 
v_{0}:=\min \left\{k_{1}, k_{2}\right\}.$$ Then the following inequality holds
\begin{equation}\label{e3.3}
c^{*}< \frac{\nu_{0}(\mu-\sigma Q) \alpha_{Q}^{\sigma(Q-1)}}{ \alpha_{0}^{\sigma(Q-1)} Q\mu
\sigma^{\sigma Q+1}}.
\end{equation}
\end{lemma}
\begin{proof}
In order to prove \eqref{e3.3}, let $\varepsilon>0$ be so small that
there exists a cut-off function $\psi_{\varepsilon}\in
C_{0}^{\infty}( B_{R}(\xi_{0}))$ such that $$
0 \leqslant
\psi_{\varepsilon} \leqslant 1,
\
\operatorname{supp}(\psi_{\varepsilon}) \subset B_{\varepsilon}(0),
\
\psi_{\varepsilon}  \equiv 1\
\hbox{on}
\
B_{\frac{\varepsilon}{2}}(0),
\
|\nabla \psi_{\varepsilon}| \leqslant \frac{4}{\varepsilon}.$$
Then
we have
$$
\|\psi_{\varepsilon}\|^{Q}=\int_{B_{\varepsilon}(0)}|\nabla
\psi_{\varepsilon}|^{Q}
d\xi+\int_{B_{\varepsilon}(0)}\psi_{\varepsilon}^{Q} d \xi\leqslant
2\left|B_{\varepsilon}(0)\right|=2\left|B_{1}(0)\right|\varepsilon^{Q}=2\varepsilon^{Q}\pi
.
$$
On the other hand, since $\eta>\eta_{1}$, we
obtain by $\left(F_{3}\right),$
\begin{equation}\label{e3.4}
I_{\lambda}(\psi_{\varepsilon},
\psi_{\varepsilon})<\frac{\mathcal{K}_{1}(2\varepsilon^{Q}\pi)+\mathcal{K}_{2}(2\varepsilon^{Q}\pi)}{Q}-\lambda
\eta_{1} \pi=0.
\end{equation}
By the definition of $\gamma(t):=(t \psi_{\varepsilon}, t \psi_{\varepsilon})$, we get the path $\gamma:[0,1] \rightarrow W_{0}^{1, Q}\left(\Omega, \mathbb{R}^{2}\right)$.
 Then  $\gamma \in \Gamma$ by \eqref{e3.4}, and we obtain
$$
\begin{aligned}
c^{*} & \leqslant \max _{t \in[0,1]}\left[\frac{1}{Q} \mathcal{K}_{1}\left(t^{Q}\|\psi_{\varepsilon}\|^{Q}\right)+\frac{1}{Q} \mathcal{K}_{2}\left(t^{Q}\|\psi_{\varepsilon}\|^{Q}\right)-\lambda \int_{\Omega} F(t \psi_{\varepsilon}, t \psi_{\varepsilon}) \mathrm{d}\xi\right] \\& \leqslant \max _{t
\in[0,1]}\left[\left(\frac{\mathcal{K}_{1}(2\varepsilon^{Q}\pi)+\mathcal{K}_{2}(2\varepsilon^{Q}\pi)}{Q}\right)
t^{Q}-\lambda \eta t^{\mu} \int_{B_{\frac{\varepsilon}{2}}(0)}
\psi_{\varepsilon}^{\mu} \mathrm{d}\xi\right]
 \leqslant \lambda \pi
\max _{t \in[0,+\infty)}\left[\eta_{1} t^{Q}-\eta t^{\mu}\right],
\end{aligned}
$$
where  $\mathcal{K}_{1}$ and $\mathcal{K}_{2}$ are convex. Consequently, we get
$$
\max _{t \in[0,+\infty)}\left[\eta_{1} t^{Q}-\eta
t^{\mu}\right]=\frac{1}{\eta ^{\frac{Q}{\mu-Q}}}(\mu-Q)
Q^{\frac{Q}{\mu-Q}}\left(\frac{\eta_{1}}{\mu}\right)^{\frac{\mu}{\mu-Q}}.$$
 Moreover, we have $$ c^{*} \leqslant \frac{\lambda
\pi}{\eta^{\frac{Q}{\mu-Q}}}(\mu-Q)
Q^{\frac{Q}{\mu-Q}}\left(\frac{\eta_{1}}{\mu}\right)^{\frac{\mu}{\mu-Q}}.
$$ Therefore \eqref{e3.3} holds  when $\eta$ satisfies
\eqref{e3.2}.
\end{proof} 
\subsection*{Acknowledgements} 
Li was supported by the  Graduate Scientific Research Project
of Changchun Normal University (SGSRPCNU [2022], Grant No. 059).
Liang was supported by the Foundation for China Postdoctoral Science
Foundation (Grant No. 2019M662220), the Research Foundation of
Department of Education of Jilin Province (Grant No.
JJKH20211161KJ), and the Natural Science Foundation of Jilin
Province (Grant no. YDZJ202201ZYTS582).  Repov\v{s} was supported by
the Slovenian Research Agency program No. P1-0292 and grants Nos.
N1-0278, N1-0114, and N1-0083. 
The authors thank the anonymous referees for their suggestions and comments.

\subsection*{Competing Interests} The authors have no competing interests to declare that are relevant to the content of this article.

\end{document}